\documentclass[reqno]{amsart}
\usepackage{amsmath}
\usepackage{amssymb}
\usepackage{amsfonts}
\usepackage{amstext}
\usepackage[cp866]{inputenc}
\usepackage[english]{babel}
\theoremstyle{plain}\newtheorem{lem}[]{Lemma}
\theoremstyle{plain}
\theoremstyle{plain}
\theoremstyle{plain}

\title[ASYMPTOTIC EXPANSION FOR DISTRIBUTION OF MARKOVIAN RANDOM MOTION]
{ASYMPTOTIC EXPANSION FOR DISTRIBUTION OF MARKOVIAN RANDOM MOTION}
\author{Anatoly A. Pogorui}
\address{Anatoly A. Pogorui, Zhitomir State University,
Velyka Berdychivska str. 40, 10008 Zhitomir, Ukraine. {\it E-mail:
pogor2@mail.ru}}


\begin{document}
\begin{abstract}
 In this paper we study an asymptotic expansion for the distribution
 of a random motion of a particle driven by a Markov process in diffusion
 approximation. We show that the singularly perturbed equation of a Markovian random
 motion can be reduced to the regularly perturbed equation for the distribution of
 the random motion.
\end{abstract}

 \maketitle

\section{Introduction}\label{s1}
The first CLT for additive functionals of a Markov chain with
noncountable phase space was proved by Doeblin \cite{db}.
Additional functionals of Markov and semi-Markov processes with
finite phase space have intensively been studied by V.S.Korolyuk,
A.F.Turbin, M.Pinsky, V.M.Shurenkov and others
\cite{kt},\cite{sh},\cite{p}.

In 1951 S. Goldstein introduced the telegrapher's stochastic
process in his seminal paper \cite{g}, which is a random motion
driven by a homogeneous Poisson process. This basic telegrapher
process has been extended by M.Kac in \cite{kc}. Goldstein-Kac's
telegraph process on the line, and its weak convergence to the
one-dimensional Brownian motion is well-known.

This paper deals with the $n$-dimensional random motion which is
an additional functional of some Markov process. This kind of
models is well known and popular in the physical literature for
the description of the long polymer molecules. For example, one of
the forms of the Airing model in \cite{k} is similar to the model
in this paper.

 Let us consider the random
motion of a particle in $\mathbb{R}^{n}$ driven by a Markov
process $\xi(t)$, which sojourn times at states are exponential
distributed with rate $\lambda >0$ and transition probabilities
$p_{ij}=\frac{1}{2n-1}\delta_{ij}$, $i,j\in E=\{1,2,\ldots,2n\}$,
where $E$ is the phase space of $\xi(t)$.

Let $\vec{b}_1,\ldots, \vec{b}_n$ be a Cartesian basis of
$\mathbb{R}^{n}$. Put $\vec{e}_{1}=\vec{b}_{1}$,
$\vec{e}_{2}=-\vec{b}_{1}$, $\vec{e}_{3}=\vec{b}_{2}$,
$\vec{e}_{4}=-\vec{b}_{2}$, ..., $\vec{e}_{2n-1}=\vec{b}_{n}$,
$\vec{e}_{2n}=-\vec{b}_{n}$ and $\vec{v}_{i}=v\vec{e}_{i}$,
$i=1,2,\ldots,2n$, where $v>0$ is constant speed of the particle.

We assume that the particle moves in $n$-dimensional space in the
following manner: If at some instant $t$ the particle has velocity
$\vec{v}_{i}$, then at a renewal moment of the Markov process the
particle takes a new velocity $\vec{v}_{j}$, $j\neq i$  with
probability $p_{ij}=\frac{1}{2n-1}$. The particle continues its
motion with velocity $\vec{v}_{j}$ until the next renewal moment
of the Markov process, and so on.

Let us denote by $\vec{r}(t)=(x_{1}(t),x_{2}(t),\ldots,x_{n}(t))$,
$t\geq 0$ the particle position at time $t$. Consider the function

$$\vec{C}(i)=(C_{1}(i),C_{2}(i),\ldots, C_{n}(i))=\vec{v}_{i},$$ $i\in E.$

Then the position of the particle at time $t$ can be expressed as
$$\vec{r}(t)=\vec{r}(0)+\int_{0}^{t}\vec{C}(\xi(t))dt.$$

\section{Equation for probability density of particle position}
 Let us consider the bivariate stochastic process
$\varsigma(t)=(\vec{r}(t),\xi(t))$ with the phase space
$\mathbb{R}^{n}\times E.$ It is well known that this process is
Markovian and the generating operator of $\varsigma(t)$ is of the
following form \cite{kt},\cite{ks}:
\begin{equation}
A\varphi(\vec{r},i)=\vec{C}(i)\varphi'(\vec{r},i)+\lambda[P\varphi(\vec{r},i)-\varphi(\vec{r},i)],
\label{one}
\end{equation}
where
$$\vec{C}(i)\varphi'(\vec{r},i)=C_{1}(i)\frac{\partial}{\partial
x_{1}}\varphi(\vec{r},i)+ C_{2}(i)\frac{\partial}{\partial
x_{2}}\varphi(\vec{r},i)+\ldots+C_{n}(i)\frac{\partial}{\partial
x_{n}}\varphi(\vec{r},i),$$ and
$P\varphi(\vec{r},i)=\frac{\lambda}{2n-1}\sum\limits_{j\in
E\setminus i}\varphi(\vec{r},i)$.

Now, let us consider the density function
\begin{eqnarray*}
f_{i}(t,x_{1},&\ldots&,x_{n})dx_{1}\ldots dx_{n}=\\
&=& P\{x_{1}\leq\ x_{1}(t)\leq x_{1}+dx_{1},\ldots, x_{n}\leq\
x_{n}(t)\leq x_{n}+dx_{n}\}
\end{eqnarray*}

It is easily verified that
$f(t,x_{1},\ldots,x_{n})=\sum\limits_{i=1}\limits^{n}f_{i}(t,x_{1},\ldots,
x_{n})$ is the probability density of the particle position in
$\mathbb{R}^n$ at time $t$.

\begin{lem}
The function $f$ satisfies the following differential equation
\begin{eqnarray*}\prod
\limits_{i=1}\limits^{2n}\{\frac{\partial}{\partial
t}+(-1)^{i}v\frac{\partial}{\partial
x_{i}}+\frac{2n\lambda}{2n-1}\}f+\\
\frac{2n\lambda}{2n-1}\sum \limits_{k=1}\limits^{2n}\prod
\limits_{i=1 \atop i\neq k}\limits^{2n}\{\frac{\partial}{\partial
t}+(-1)^{i}v\frac{\partial}{\partial
x_{i}}+\frac{2n\lambda}{2n-1}\}f=0
\end{eqnarray*}
\end{lem}

\begin{proof}
For $i\in E$ function $f_{i}$ satisfies the first Kolmogorov
equation, namely
\begin{equation}
\frac{\partial f_{i}(t,x_{1},\ldots, x_{n})}{\partial
t}=Af_{i}(t,x_{1},\ldots, x_{n}), \quad i\in E  \label{two}
\end{equation}
with initial conditions $
f_{i}(0,x_{1},\ldots,x_{n})=f_{i}^{(o)}$.

Eq.(\ref{two}) can be written in more detail as follows
\begin{eqnarray*}
\frac{\partial f_{i}(t,x_{1},\ldots, x_{n})}{\partial t}+
(-1)^{i}v\frac{\partial}{\partial x_{i}}f_{i}(t,x_{1},\ldots,
x_{n})+\lambda f_{i}(t,x_{1},\ldots, x_{n})-
\end{eqnarray*}
\begin{equation}
\frac{\lambda}{2n-1} \sum\limits_{j \in E \setminus i}
f_{i}(t,x_{1},\ldots, x_{n})=0,\quad i\in E. \label{three}
\end{equation}

Now, put $\vec{f}(t,x_{1},\ldots, x_{n})=\{f_{i}(t,x_{1},\ldots,
x_{n}), i\in E\}$. The set of equations (3) can be written in the
following form $$L_{2n}\vec{f}=0,$$ where
$L_{2n}=\{l_{ij}\}_{ij\in E},$ $l_{ii} = \frac{\partial}{\partial
t}+(-1)^{i}v\frac{\partial}{\partial x_{i}}+\lambda$, $l_{ij} =
\frac{-\lambda}{2n-1}$, $i\neq j.$

 The function
 $f$ satisfies the following equation \cite{kot}
\begin{equation}
\det (L_{2n})f=0, \label{four}
\end{equation}
with the initial condition
$f(0,x_{1},\ldots,x_{n})=\sum\limits_{k=1}\limits^{n}f_{i}^{(o)}.$

The determinant of matrix $L_{2n}$ is well known and it has the
form
\begin{eqnarray*}\det (L_{2n})=\prod
\limits_{i=1}\limits^{2n}\{\frac{\partial}{\partial
t}+(-1)^{i}v\frac{\partial}{\partial
x_{i}}+\frac{2n\lambda}{2n-1}\}+\\
\frac{2n\lambda}{2n-1}\sum \limits_{k=1}\limits^{2n}\prod
\limits_{i=1 \atop i\neq k}\limits^{2n}\{\frac{\partial}{\partial
t}+(-1)^{i}v\frac{\partial}{\partial
x_{i}}+\frac{2n\lambda}{2n-1}\}
\end{eqnarray*}\end{proof}
Since $v\frac{\partial}{\partial x_{i}} $ and
$-v\frac{\partial}{\partial x_{i}} $  appear in $L_{2n}$
symmetrically it is easy to see that all monomials of the
polynomial $\det (L_{2n})$  contain $v^{k}$ only with even powers
$k\geq 0$.

\section{Reduction of singularly perturbed evolution equation to regularly perturbed
     equation
}

Let us put $v=\varepsilon ^{-1}$ and $\lambda=\varepsilon ^{-2}$,
where $\varepsilon >0 $ is a small parameter. It is well known
\cite{ks},\cite{s},\cite{t} that the solution of Eq.(\ref{three})
in hydrodynamical limit (as $\varepsilon \rightarrow 0 $) weakly
converges to the corresponding functional of Wiener process.

By using technique developed in \cite{vb}, we can find asymptotic
expansion of the solution of Eq.(\ref{three}), which consists of
regular and singular terms. This technique involves tedious
calculations \cite{kt},\cite{s}.

{\bf Proposition.} {\it The equation $\det (L_{2n})f=0$ is regular
perturbed, that is, multiplying it by $\varepsilon^{4n-2}$, we get
\begin{equation*}
\frac{\partial}{\partial t}f=\frac{2n-1}{2n^{2}}(\frac{\partial
^{2}}{\partial x_{1}^{2}}+\ldots +\frac{\partial ^{2}}{\partial
x_{n}^{2}})f+D_{\varepsilon}f,
\end{equation*}
where $D_{\varepsilon}=\varepsilon ^{2}D_{1}+\varepsilon
^{4}D_{2}+\ldots$, $D_{i},$ $i=1,2,\ldots,$ are respective
differential operators.}

\begin{proof}
To avoid cumbersome expressions, we consider the case when $n=3 $.
Let us put $x=x_{1}$, $y=x_{2}$, $z=x_{3}$. In this case
Eq.(\ref{three}) has the following form
\begin{equation}
\frac{\partial f_{i}(t,x,y,z)}{\partial t}+
(-1)^{i}v\frac{\partial}{\partial x}f_{i}(t,x,y,z)+\lambda
f_{i}(t,x,y,z)-\frac{\lambda}{5}\sum\limits_{j\in E \atop j\neq
i}f_{i}(t,x,y,z), \label{five}
\end{equation}
$i=1,\ldots,6$.

Putting  $v=\varepsilon ^{-1}$ and $\lambda=\varepsilon ^{-2}$, we
obtain the following singularly perturbed system of equations
\begin{eqnarray*}
\frac{\partial f_{i}(t,x,y,z)}{\partial t}+ (-1)^{i}\varepsilon
^{-1}\frac{\partial}{\partial x}f_{i}(t,x,y,z)+\varepsilon ^{-2}
f_{i}(t,x,y,z)-
\end{eqnarray*}
\begin{eqnarray}
\varepsilon ^{-2}\frac{1}{5}\sum\limits_{j\in E\atop j\neq
i}f_{i}(t,x,y,z)=0, \label{six}
\end{eqnarray}
$i=1,\ldots,n.$

Let us consider the equation $\det(L_{6})f=0,$ where
$f(t,x,y,z)=\sum\limits_{i=1}\limits^{6}f_{i}(t,x,y,z).$  It is
easy to see that elements of matrix $L_{6}=(l_{ij})_{i,j\in E}$
are as follows: $ l_{ii}= \frac{\partial}{\partial
t}+(-1)^{i}v\frac{\partial}{\partial t}+\lambda$ and
$l_{ij}=\frac{\lambda}{5}$, for $i\neq j$, $i,j\in
\{1,2,\ldots,6\}.$

Hence, the equation $\det(L_{6})f=0$ has the following form
$$
\det(L_{6})f(t,x,y,z)=\{\frac{7776}{3125}\varepsilon^{-10}
\frac{\partial}{\partial t}-\frac{432}{625}\varepsilon^{-10}
\left( \frac{\partial^{2}}{\partial
x^{2}}+\frac{\partial^{2}}{\partial
y^{2}}+\frac{\partial^{2}}{\partial z^{2}} \right)-$$
$$\frac{432}{125}\varepsilon^{-8} \left(
\frac{\partial^{2}}{\partial x^{2}}+\frac{\partial^{2}}{\partial
y^{2}}+\frac{\partial^{2}}{\partial z^{2}} \right)
\frac{\partial}{\partial t} -\frac{144}{25}\varepsilon^{-6} \left(
\frac{\partial^{2}}{\partial x^{2}}+\frac{\partial^{2}}{\partial
y^{2}}+\frac{\partial^{2}}{\partial z^{2}}
\right)\frac{\partial^{2}}{\partial t^{2}}+$$

$$\frac{1296}{125}\varepsilon^{-2}\frac{\partial^{2}}{\partial
t^{2}}+\frac{24}{25}\varepsilon^{-8} \left(
\frac{\partial^{4}}{\partial x^{2}\partial
y^{2}}+\frac{\partial^{2}}{\partial y^{2}\partial
z^{2}}+\frac{\partial^{2}}{\partial x^{2}\partial
z^{2}}\right)+\frac{72}{5}\varepsilon^{-4}\frac{\partial^{4}}{\partial
t^{4}}-$$ $$4\varepsilon^{-4}\left( \frac{\partial^{2}}{\partial
x^{2}}+\frac{\partial^{2}}{\partial
y^{2}}+\frac{\partial^{2}}{\partial z^{2}}
\right)\frac{\partial^{3}}{\partial t^{3}}+2\varepsilon^{-6}
\left( \frac{\partial^{4}}{\partial x^{2}\partial
y^{2}}+\frac{\partial^{2}}{\partial y^{2}\partial
z^{2}}+\frac{\partial^{2}}{\partial x^{2}\partial
z^{2}}\right)\frac{\partial}{\partial t}+$$
$$6\varepsilon^{-2}\frac{\partial^{5}}{\partial t^{5}}+
\varepsilon^{-8}\left( \frac{\partial^{4}}{\partial x^{2}\partial
y^{2}}+\frac{\partial^{2}}{\partial y^{2}\partial
z^{2}}+\frac{\partial^{2}}{\partial x^{2}\partial
z^{2}}\right)\frac{\partial^{2}}{\partial t^{2}}-
\varepsilon^{-6}\frac{\partial^{6}}{\partial x^{2}\partial
y^{2}\partial z^{2}}-$$

\begin{eqnarray}
\varepsilon^{-2}\left( \frac{\partial^{2}}{\partial
x^{2}}+\frac{\partial^{2}}{\partial
y^{2}}+\frac{\partial^{2}}{\partial z^{2}}
\right)\frac{\partial^{4}}{\partial
t^{4}}+\frac{\partial^{6}}{\partial
t^{6}}+\frac{432}{25}\varepsilon^{-6}\frac{\partial^{3}}{\partial
t^{3}}\}f(t,x,y,z)=0, \label{seven}
\end{eqnarray}
with the initial condition $f(0,x,y,z)=f_{0}.$

Multiplying Eq.(\ref{seven}) by $\varepsilon^{10} $, we obtain
$$
\{\frac{\partial}{\partial
t}-\frac{5}{18}\Delta+\varepsilon^{2}\frac{25}{6}\left[\frac{\partial^{2}}{\partial
t^{2}}-\frac{1}{3}\Delta\frac{\partial}{\partial
t}+\frac{5}{54}\Delta^{(2)}\right]+
$$

$$
\varepsilon^{4}\frac{125}{18}\left[\frac{\partial^{3}}{\partial
t^{3}}-\frac{1}{3}\Delta\frac{\partial^{2}}{\partial
t^{2}}+\frac{5}{216}\Delta^{(2)}\frac{\partial}{\partial
t}\right]$$

$$+\varepsilon^{6}\frac{625}{216}\left[\frac{\partial^{4}}{\partial
t^{4}}-\frac{5}{18}\Delta\frac{\partial^{3}}{\partial
t^{3}}+\frac{5}{72}\Delta^{(2)}\frac{\partial^{2}}{\partial
t^{2}}\right]+
$$

\begin{equation}
\varepsilon^{8}\frac{3125}{1296}\left[\frac{\partial^{5}}{\partial
t^{5}}-\frac{1}{6}\Delta^{(2)}\frac{\partial^{4}}{\partial
t^{4}}\right]+
\varepsilon^{10}\frac{3125}{7776}\frac{\partial^{6}}{\partial
t^{6}}\}f=0, \label{eight}
\end{equation}
\\
where $\Delta = \frac{\partial^{2}}{\partial
x^{2}}+\frac{\partial^{2}}{\partial
y^{2}}+\frac{\partial^{2}}{\partial z^{2}}$, $\Delta^{(2)} =
\frac{\partial^{4}}{\partial x^{2}\partial
y^{2}}+\frac{\partial^{4}}{\partial y^{2}\partial
z^{2}}+\frac{\partial^{4}}{\partial x^{2}\partial z^{2}}.$

\end{proof}

\begin{lem} The solution of equation (\ref{eight}) with the initial condition
$f(0,x,y,z)=u_{0}(0,x,y,z)+\varepsilon^{2}u_{1}(0,x,y,z)+
\varepsilon^{4}u_{2}(0,x,y,z)+\ldots$, has asymptotic expansion
\begin{equation}
f(t,x,y,z)=u_{0}(t,x,y,z)+\varepsilon^{2}u_{1}(t,x,y,z)+
\varepsilon^{4}u_{2}(t,x,y,z)+\ldots, \label{nine}
\end{equation}
where the principal term $u_{0}(t,x,y,z)$ represents the solution
of equation
$$
\frac{\partial}{\partial t}u_{0}(t,x,y,z)=\frac{5}{18}\Delta
u_{0}(t,x,y,z).$$

\end{lem}

\begin{proof}
To find the asymptotic expansion of the solution of (\ref{eight})
we use the method proposed in \cite{vb} and developed in
\cite{kt}. In conformity with this method the solution of
(\ref{eight}) can be expanded into the series (\ref{nine}), where
$\varepsilon>0$ is small.

Substituting (\ref{nine}) into (\ref{eight}), we get the following
equations for computing of $u_{i}$, $i\geq 0$
$$
\frac{\partial}{\partial t}u_{0}(t,x,y,z)=\frac{5}{18}\Delta
u_{0}(t,x,y,z);$$

$$
\frac{\partial}{\partial t}u_{1}(t,x,y,z)=\frac{5}{18}\Delta
u_{1}(t,x,y,z)+$$
$$\varepsilon^{2}\frac{25}{6}\left[\frac{\partial^{2}}{\partial
t^{2}}-\frac{1}{3}\Delta\frac{\partial}{\partial
t}+\frac{5}{54}\Delta^{(2)}\right]u_{0}(t,x,y,z);$$

$$
\frac{\partial}{\partial t}u_{2}(t,x,y,z)=\frac{5}{18}\Delta
u_{2}(t,x,y,z)+$$
$$\varepsilon^{2}\frac{25}{6}\left[\frac{\partial^{2}}{\partial
t^{2}}-\frac{1}{3}\Delta\frac{\partial}{\partial
t}+\frac{5}{54}\Delta^{(2)}\right]u_{1}(t,x,y,z)+$$

$$\frac{125}{18}\left[\frac{\partial^{3}}{\partial
t^{3}}-\frac{1}{3}\Delta\frac{\partial^{2}}{\partial
t^{2}}+\frac{5}{216}\Delta^{(2)}\frac{\partial}{\partial
t}\right]u_{0}(t,x,y,z);$$

$$\ldots$$

$$
\frac{\partial}{\partial t}u_{m+5}(t,x,y,z)=\frac{5}{18}\Delta
u_{m+5}(t,x,y,z)+
$$

$$\frac{25}{6}\left[\frac{\partial^{2}}{\partial
t^{2}}-\frac{1}{3}\Delta\frac{\partial}{\partial
t}+\frac{5}{54}\Delta^{(2)}\right]u_{m+4}(t,x,y,z)+
$$

$$\frac{125}{18}\left[\frac{\partial^{3}}{\partial
t^{3}}-\frac{1}{3}\Delta\frac{\partial^{2}}{\partial
t^{2}}+\frac{5}{216}\Delta^{(2)}\frac{\partial}{\partial
t}\right]u_{m+3}(t,x,y,z)+$$

$$\frac{625}{216}\left[\frac{\partial^{4}}{\partial
t^{4}}-\frac{5}{18}\Delta\frac{\partial^{3}}{\partial
t^{3}}+\frac{5}{72}\Delta^{(2)}\frac{\partial^{2}}{\partial
t^{2}}\right]u_{m+2}(t,x,y,z)+$$

$$\frac{3125}{1296}\left[\frac{\partial^{5}}{\partial
t^{5}}-\frac{1}{6}\Delta^{(2)}\frac{\partial^{4}}{\partial
t^{4}}\right]u_{m+1}(t,x,y,z)+\frac{3125}{7776}\frac{\partial^{6}}{\partial
t^{6}}u_{m}(t,x,y,z)=0,$$ for $m\geq 0.$
\end{proof}

Let us consider the function
$$\widetilde{f}_{k}^{(\varepsilon)}(t,x,y,z)=u_{0}(t,x,y,z)+\varepsilon^{2}u_{1}(t,x,y,z)+
\ldots+\varepsilon^{2k}u_{k}(t,x,y,z).$$

In \cite{s}  for the solution of a singularly perturbed equation
of type (\ref{three}) the remainder of asymptotic expansion in the
circuit of diffusion approximation was studied.

Taking into account that $f(t,x,y,z)=\sum \limits_{i\in E}
f_{i}(t,x,y,z)$, it follows from the estimate of the reminder in
\cite{kt},\cite{aks} that
$$\parallel f(t,x,y,z)-\widetilde{f}^{(\varepsilon)}(t,x,y,z)\parallel =
O(\varepsilon^{2k})$$

\section{Conclusions}
Singularly perturbed equations of type (\ref{three}) in
hydrodynamical limit (where $\frac{\lambda}{v^{2}}=O(1), \lambda
\downarrow 0$) have become the subject of a great deal of
researches \cite{kc},\cite{kt},\cite{p},\cite{s},\cite{t} and
others. By using the technique of professor A.F.Turbin \cite{kot},
we reduce singularly perturbed system of equations (\ref{three})
to Eq.(\ref{four}), which turns out to be regularly perturbed in
hydrodynamical limit to diffusion process.

Therefore, in such cases we may simplify cumbersome calculations
of terms of asymptotic expansion for solution of singulary
perturbed equations for functionals of Markovian random motion.

\vspace{0.5 cm}

{\bf Acknowledgements} I wish to thank an anonymous referee for
careful reading of the paper and insightful comments and remarks.



\begin{thebibliography}{GSS}

\bibitem[1]{db} Doeblin W. Sur les propietes asymptotiques de
mouvement regis par certains tupes de chaines simples.
\textit{Bull. Math. Soc. Raum. Sci.\/} v.39, No.1, 57--115; No.2,
3--61 (1938)

\bibitem[2]{g} Goldstein S. On diffusion by discontinuous movements and on the
telegraph equation. \textit{Quart. J. Math. Mech.\/} 4, 129--156
(1951).

\bibitem[3]{kc} Kac M. \textit{Probability and Related Topics in Physical Sciences.\/}
Interscience Publishers, New York, (1959).

\bibitem[4]{kt}  Korolyuk V. S.,  Turbin A. F. \textit{Mathematical Foundations of
the State Lumping of Large Systems.\/} Kluwer Academic Publishers,
(1994).

\bibitem[5]{sh} Shurenkov V.M. \textit{Ergodic Theory of Random
Processes.\/} Nauka, Moscow, (1989). (in Russian)

\bibitem[6]{p}  Pinsky M.A. \textit{Lectures on Random Evolution},
World Scientific Publishing, New Jersey, (1991).

\bibitem[7]{k} Kac M. A stochastic model related to the telegrapher's equation.
\textit{Rocky Mountain J. Math.\/} 4, 497--509 (1974).

\bibitem[8]{ks} Korolyuk V. S., Swishchuk A. V. \textit{Semi-Markov Random Evolutions.\/}
Kluwer Academic Publishers (1995).

\bibitem[9]{kot}  Kolesnik A. D., Turbin A. F. The equation of symmetric Markovian random
evolution in a plane. \textit{Stoc. Proc. Appl.\/}  75, 67--87
(1998).

\bibitem[10]{s} Samoilenko I.V. Asymptotic expansion for the functional of markovian
evolution in $R^{d}$ in the circuit of diffusion approximation.
\textit{Journal of Applied Mathematics and Stochastic Analysis\/}
2005:3, 247--257, (2005).

\bibitem[11]{t} Turbin, A.F. Mathematical model of one-dimensional Brownian motion as
alternative to mathematical model of Einstein, Wiener, Levy.
\textit{Fractal Analysis and Related Fields\/} 2, 47-60 (1998).
(in Russian)

\bibitem[12]{aks}  Albeverio S., Koroliuk V., Samoilenko I. Asymptotic Expansion of Semi-Markov
Random Evolutions, \textit{Bonn, Preprint no. 277\/} (2006).

\bibitem[13]{vb}  Vasiljeva A.B., Butuzov V.F. \textit{Asymptotic methods in the theory of
singular perturbations.\/} Vysshaja shkola, Moscow, (1990). (in
Russian)


\end{thebibliography}
\end{document}